\documentclass[a4paper]{article}
\usepackage[english]{babel}
\usepackage[T1]{fontenc}
\usepackage[utf8]{inputenc}
\usepackage[obeyspaces]{url}
\usepackage{mathtools, amsmath, amsthm, amssymb, mathtools,
  mathabx, tikz, 
  subcaption, framed, float}
\usepackage[
  colorlinks,
  citecolor=brown,
  linkcolor=red,
  urlcolor=blue,
]{hyperref}
\usepackage[affil-it]{authblk}
  \usepackage{bm}
\usepackage{cleveref}

\newtheorem{cor}{Corollary}
\newtheorem{thm}{Theorem}

\providecommand{\keywords}[1]{\textbf{Keywords:} #1}
\newcommand{\m}[1]{\mathrm{#1}}
\newcommand\sloane[1]{\href{https://oeis.org/#1}{#1}}

\newcommand\e{\bullet}

\author{Jean-Luc Baril \thanks{Electronic address: \texttt{barjl@u-bourgogne.fr}}}
\author{Richard Genestier \thanks{Electronic address: \texttt{Richard.Genestier@u-bourgogne.fr}}}
\author{Sergey Kirgizov \thanks{Electronic address: \texttt{sergey.kirgizov@u-bourgogne.fr}; Corresponding author}}

\affil{LIB, Univ. Bourgogne Franche-Comté, France}

\date{\today}

\title{Pattern distributions in Dyck paths with a first return decomposition constrained by height}

\title{Pattern distributions in Dyck paths with a first return decomposition constrained by height}

\begin{document}

\maketitle

\begin{abstract}
  We provide generating functions for the popularity and the distribution of patterns of length at most three over the set of Dyck paths having a first return decomposition constrained by height.
\end{abstract}
\section{Introduction and notation}
\label{introduction}
 Dyck paths with a constrained first return decomposition were introduced in~\cite{decreasing-dyck} where the authors present both enumerative results using generating functions and a constructive bijection with the set of Motzkin paths. In~\cite{decreasing-motzkin}, a similar study has been conducted for Motzkin, 2-colored Motzkin, Schröder and Riordan paths. In the literature, many papers deal with the enumeration of classical Dyck paths according to
different parameters, e.g. length, number of peaks, valleys, double
rises and other pattern
occurrences~\cite{Deu,Man2,Man1,Mer,Pan,Pea,Sap,Sun}.
Restricted classes of Dyck paths have also been considered, for instance Barcucci et al.~\cite{Barc} consider
Dyck paths having a non-decreasing
height sequence of valleys (see also \cite{Cza,Den}).  Other
papers deal with Motzkin paths using similar
methods~\cite{BarMot,BrenMav,DonSha,DraGan,ManSchSun,SapTsi,WagPro}.
 Motzkin and Catalan numbers appear alongside in many situations~\cite{DonSha} and several one-to-one correspondences exist between restricted Dyck paths
 and Motzkin paths.
For instance, Dyck paths avoiding a triple rise are
enumerated by the Motzkin numbers~\cite{callan2004two}.

  In this paper, we focus on the distribution and the popularity of  patterns of length at most three in constrained Dyck paths defined in~\cite{decreasing-dyck}. Our method consists in showing how patterns are
getting transferred from constrained Dyck paths to Motzkin paths, which settles us in a more suitable ground in order to provide  generating functions for the distribution and the popularity.

 A {\em Motzkin path} of length $n\geq 0$ is a lattice path consisting of
flat steps $F=(1,0)$, up steps $U=(1,1)$ and down steps $D=(1,-1)$,
starting at $(0,0)$, ending at $(n,0)$ and never going below the
$x$-axis. For $n\geq 0$, we denote by $\mathcal{M}_n$ the set of all Motzkin
paths of length $n$ and  we set $\mathcal{M}=\bigcup_{n\geq
  0}\mathcal{M}_n$. A Motzkin path of length $2n$
  with no flat steps is a {\em Dyck path} of semilength $n$. For $n\geq 0$, let $\mathcal{D}_n$ be the set of all Dyck paths of semilength $n$ and $\mathcal{D}=\bigcup_{n\geq 0}\mathcal{D}_n$. The cardinality of
$\mathcal{D}_n$ is given by the $n$th Catalan number $c_n=\frac{1}{n+1} {{2n}\choose{n}}$, which is the
general term of the sequence
\sloane{A000108} in the On-line Encyclopedia
of Integer Sequences of N.J.A. Sloane~\cite{sloane}.
The cardinality of $\mathcal{M}_n$ is given by the
$n$th Motzkin number $\sum_{k=0}^{\lfloor n/2\rfloor} \binom{n}{2k}
c_k$ (see \sloane{A000108} in~\cite{sloane}).

Any non-empty Dyck path $P\in\mathcal{D}$ has a unique first return
decomposition~\cite{Deu} of the form $P=U\alpha D\beta$ where $\alpha$
and $\beta$ are two Dyck paths in
$\mathcal{D}$. In
 \cite{decreasing-dyck}, the authors introduced the set $\mathcal{D}^{h,\geq}$ constituted of the empty Dyck path and the Dyck paths  in $\mathcal{D}$ having a first return decomposition  satisfying $$h(U\alpha D)\geq h(\beta)$$ where $\alpha,\beta\in \mathcal{D}^{h,\geq}$
   and $h$ returns the maximal
height of a Dyck path. For $n\geq 0$, let  $\mathcal{D}_n^{h,\geq}$ be the subset of Dyck paths of semilength $n$ in  $\mathcal{D}^{h,\geq}$. For instance,  $\mathcal{D}_3^{h,\geq}$
consists of  four Dyck paths $UDUDUD$, $UUDDUD$, $UUDUDD$ and
$UUUDDD$. In~\cite{decreasing-dyck},  the authors prove, using generating functions,  that $\mathcal{D}_n^{h,\geq}$ and $\mathcal{M}_n$ have the same cardinality, and they  present also the following bijection $\phi$ between these sets.

For
$P\in \mathcal{D}^{h,\geq}$,

$$\phi(P)=\left\{\begin{array}{ll}
\epsilon& \mbox{ if } P=\epsilon,\\
\phi(\alpha) F&\mbox{ if } P=\alpha UD,\\
\phi(\alpha)\phi(\gamma)U\phi(\beta) D & \mbox{ if } P= \alpha UU\beta D \gamma D.\\
\end{array}\right.$$

  For instance, the images by $\phi$ of $UDUDUD$, $UUDDUD$, $UUDUDD$, $UUUDDD$, $UUUUDDDDUUUDDUDD$
 are respectively $FFF$, $UDF$, $FUD$, $UFD$ and $UUDDFUFD$. We refer
 to Figure \ref{fig2} for an illustration of this mapping.

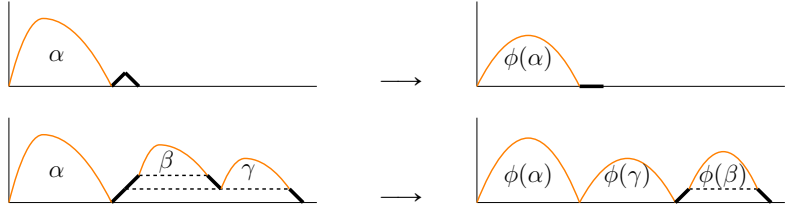
\begin{figure}[h]
\begin{center}\scalebox{0.45}{\begin{tikzpicture}[ultra thick]
 \draw[black, thick] (0,0)--(9,0); \draw[black, thick] (0,0)--(0,2.5);
  \draw[black, line width=3pt] (3,0)--(3.4,0.4);
 \draw[black, line width=3pt] (3.4,0.4)--(3.8,0);
 \draw[orange,very thick] (0,0) parabola bend (1,2) (3,0);
 \draw  (1.4,0.9) node {\huge $\mathbf{\alpha}$};
 \end{tikzpicture}}
 \qquad$\longrightarrow$\qquad
\scalebox{0.45}{\begin{tikzpicture}[ultra thick]
 \draw[black, thick] (0,0)--(9,0); \draw[black, thick] (0,0)--(0,2.5);
 \draw[black, line width=3pt] (3,0)--(3.7,0);
\draw[orange,very thick] (0,0) parabola bend (1.5,1.5) (3,0);
 \draw  (1.5,0.8) node {\huge $\mathbf{\phi(\alpha)}$};
 \end{tikzpicture}}
\end{center}
\begin{center}\scalebox{0.45}{\begin{tikzpicture}[ultra thick]
 \draw[black, thick] (0,0)--(9,0); \draw[black, thick] (0,0)--(0,2.5);
  \draw[black, line width=3pt] (3,0)--(3.4,0.4);\draw[black, dashed, very thick] (3.8,0.8)--(5.8,0.8);
  \draw[black, dashed, very thick] (3.4,0.4)--(8.2,0.4);
  \draw[black, line width=3pt] (3.4,0.4)--(3.8,0.8);\draw[black, line width=3pt] (5.8,0.8)--(6.2,0.4);
  \draw[black, line width=3pt] (8.2,0.4)--(8.6,0);
 \draw[orange,very thick] (0,0) parabola bend (1,2) (3,0);
  \draw[orange,very thick] (3.8,0.8) parabola bend (4.4,1.7) (5.8,0.8);
  \draw[orange,very thick] (6.2,0.4) parabola bend (6.9,1.3) (8.2,0.4);
 \draw  (1.4,0.9) node {\huge $\mathbf{\alpha}$};\draw  (4.6,1.2) node {\huge $\mathbf{\beta}$};
 \draw  (7,0.9) node {\huge $\mathbf{\gamma}$};
 \end{tikzpicture}}
 \qquad$\longrightarrow$\qquad
\scalebox{0.45}{\begin{tikzpicture}[ultra thick]
 \draw[black, thick] (0,0)--(9,0); \draw[black, thick] (0,0)--(0,2.5);
 \draw[black, dashed, very thick] (6.2,0.4)--(8.2,0.4);
  \draw[black, line width=3pt] (8.2,0.4)--(8.6,0);
  \draw[orange,very thick] (0,0) parabola bend (1.5,1.9) (3,0);
   \draw[orange,very thick] (3,0) parabola bend (4.4,1.3) (5.8,0);
  \draw[black, line width=3pt] (5.8,0)--(6.2,0.4);
  \draw[orange,very thick] (6.2,0.4) parabola bend (7.2,1.5) (8.2,0.4);
 \draw[very thick]  (1.5,0.8) node {\huge $\mathbf{\phi(\alpha)}$};\draw  (4.4,0.8) node {\huge $\mathbf{\phi(\gamma)}$};
 \draw  (7.2,0.8) node {\huge $\mathbf{\phi(\beta)}$};
 \end{tikzpicture}}
\end{center}
\caption{Illustration of the bijection $\phi$ between $\mathcal{D}_n^{h,\geq}$ and $\mathcal{M}_n$.}
\label{fig2}\end{figure}

\bigskip

A {\em statistic} on a set $S$ of paths is an association of an integer in $\mathbb{Z}$ to each path in $S$. For instance the map that returns the number of steps is a statistic, and we denote by $\m 1$ (resp. $\m 0$) the constant statistic that sends any path to 1 (resp. $0$). Let $\mathcal{S}$ be the set of all statistics on a set $S$. For $\m X,\m Y\in\mathcal{S}$,  we define the statistic $\m X+\m Y$ so that $(\m X+\m Y)(P)=\m X(P)+\m Y(P)$ for any $P\in S$, which endows $\mathcal{S}$ with a $\mathbb{Z}$-module structure. Let $S$ and $T$ be two sets of paths, and let $\mathcal{S}$ and $\mathcal{T}$ be the associated statistic sets. Two statistics $\m X\in \mathcal{S}$ and $\m Y\in\mathcal{T}$ have the {\em same distribution} if and only if there exists a bijection $f$ from $S$ to $T$ such that for any $P\in S$ we have $\m X(P)=\m Y(f(P))$. In this case, we say that  $f$ transports the statistic $\m X$ into  $\m Y$, which can be shortly written with the statistic equation $f(\m X)=\m Y$ (or $\m X=\m Y$ whenever $f$ is the identity).

A {\em pattern} $X$ of length $k\geq 1$ occurs in a path $P$ if and only if $P$ contains $X$ as a sequence of consecutive steps. Note that other variants of pattern definition exist in the literature (see for instance \cite{Bach}). From a given pattern $X$ and a set $S$ of paths, we associate the {\em statistic} $\m X$ from  $S$ to $\mathbb{N}$ such that $\m X(P)$ is the number of occurrences of the pattern $X$ in $P$. The {\em popularity} of a pattern $X$ in  $S$ is the total number of occurrences of the pattern $X$ in all paths $P\in S$, that is $\sum_{P\in S} \m X(P)$.


For instance, if $P = UUUDDUDD$ then we have $\m{UD}(P) = \m{UU}(P)= 2$, $\m{DDD}(P) = 0$. Moreover, if $S=\{UUDD,UDUD\}$ then the popularity of the pattern $UD$ in $S$ is $3$. Also, for any subset $S$ of $\mathcal{M}$, we have the statistic equation $\m U=\m D$ since for any path in $S$ the number of up steps equals to the number of down steps. If we restrict $S$ to $\mathcal{D}_n$ (resp. $\mathcal{M}_n$), then we have $\m U+\m D=2\m n$ (resp. $\m U+\m F+\m D=\m n$), where $\m n$ is the  constant statistic $P\mapsto n$.

Considering the above bijection $\phi$ from  $\mathcal{D}_n^{h, \geq}$ to $\mathcal{M}_n$,  the length of $\phi(P)$ is the semilength of $P$ and we easily deduce the statistic equation:
 \begin{equation}
  \phi(\m U) = \phi(\m D) = \m U + \m D + \m F = \m n.
  \label{eq3}
\end{equation}

In this paper, we present statistic equations showing how $\phi$: $\mathcal{D}_n^{h, \geq}\rightarrow \mathcal{M}_n$ transports statistics associated to patterns of length at most three into  linear combinations of other statistics. Then, this allows us to conduct our enumerative study in the  more natural and simpler context of Motzkin paths, while a direct study on $\mathcal{D}_n^{h, \geq}$ is complicated due to the lack of an adequate recursive decomposition. So, we use statistic equations in order to derive bivariate generating functions for the distribution of patterns of length at most three, and we deduce the pattern popularity thereafter.  Section 2 deals with patterns of length 2 while Section 3 deals with patterns of length 3. Many of the resulting sequences correspond to existing entries
in the On-line Encyclopedia of Integer Sequences of
N.J.A. Sloane~\cite{sloane} enumerating a quantity of already known
combinatorial structures,  and we also obtain new sequences not yet known in~\cite{sloane}.


\section{Patterns of length 2}
\label{2}
In this part, we provide statistic equations showing how patterns of length two behave through $\phi: \mathcal{D}_n^{h,\geq}\rightarrow \mathcal{M}_n$, which allows us to deduce generating functions for the distribution and the popularity of such patterns. See Table~\ref{tab.UD} for an illustration of the distributions and Table~\ref{tab.pop2} for the first terms of popularity sequences.

\begin{thm}
  For $n\geq 0$, the bijection $\phi$ from $\mathcal{D}_n^{h,\geq}$ to $\mathcal{M}_n$ transports statistics associated to patterns of length two as follows: \begin{align}
    \phi(\m{UD}) & = \m F + \m{UD},
    \label{UD}
    \\
    \phi(\m{UU}) = \phi(\m{DD}) & = \m U + \m{UU} + \m{UF},
    \label{UU}
    \\
    \phi(\m{DU}) & = \m{FF} + \m{FU} + \m{DF} + \m{DU}.
    \label{DU}
  \end{align}
  \label{thm1}
\end{thm}

\begin{proof}
  For equation~(\ref{UD}), we refer to~\cite{decreasing-dyck}.

  For equation~(\ref{UU}), we observe that for any Dyck path of semilength $n$ an up step $U$ is always followed by $U$ or $D$, which implies the equality $\m{UU}+\m{UD}=\m{n}$ on~$\mathcal{D}_n^{h,\geq}$. With a similar argument and $\m{U}=\m{D}$, we have $\m{F} + \m{U} + \m{UD} + \m{UU} + \m{UF} = \m{n}$ on $\mathcal{M}_n$. Using~\cref{UD} we obtain $\phi(\m{UU}) = \m{n} - \phi(\m{UD}) = \m{n}
  - \big( \m{F} + \m{UD} \big)  = \m{U} + \m{UU} + \m{UF}$.

 For equation~(\ref{DU}), we observe that the  statistic $\m{UU}+\m{UD}+\m{DU}+\m{DD}$ equals $2\m{n}-\m{1}$ on $\mathcal{D}_n^{h,\geq}$. Using the straightforward equality $\m{UU}=\m{DD}$ on $\mathcal{D}_n$, we obtain $\m{DU}=2\m{n}-\m{1}-2\m{UU}-\m{UD}$.
   Applying the bijection $\phi$ and using~\cref{UD,UU} we obtain
  $\phi(\m{DU}) = 2\m{n} - \m{1} - 2 (\m{U} + \m{UU} + \m{UF}) - (\m{F} +\m{UD})$.  On the other hand, on the set of Motzkin paths of length $n$ we have the statistic equations $\m{U}+\m{F}+\m{D}=\m{n}$ and  $\m{FD}+\m{FU} + \m{FF} + \m{DD} + \m{DU} + \m{DF} + \m{UD} + \m{UU} + \m{UF}=\m{n}-{1}$, which induces
  $2\m n-\m{1}=\m{U}+\m{F}+\m{D}+\m{FD}+\m{FU} + \m{FF} + \m{DD} + \m{DU} + \m{DF} + \m{UD} + \m{UU} + \m{UF}$. Also, for any  Motzkin path we have $\m{U}=\m{D}$ which  implies $\m{UU}+\m{UF}+\m{UD}=\m{UD}+\m{FD}+\m{DD}$, and thus $\m{UU}+\m{UF}=\m{FD}+\m{DD}$. So, combining all these equations, we obtain $\phi(\m{DU})=\m{U}+\m{F}+\m{D}+\m{FD}+\m{FU} + \m{FF} + \m{DD} + \m{DU} + \m{DF} + \m{UD} + \m{UU} + \m{UF} -\m{U}-\m{D}-\m{UU}-\m{UF}-\m{FD}-\m{DD}- \m{F}-\m{UD}= \m{FF}+\m{FU}+\m{DF}+\m{DU}$.
 \end{proof}

\begin{thm} The bivariate generating functions $F_{p}(x,y)$ where the coefficient of $x^ny^k$ is the number of Dyck paths in $\mathcal{D}_n^{h,\geq}$  containing exactly $k$ occurrences of the pattern $p\in \{UD, UU, DD, DU\}$ are

  \begin{equation}
    \nonumber
  F_{UD}(x,y) = \frac{ x^{2} - x^{2} y - x y + 1 - \sqrt{- 4 x^{2} + \left(x^{2} \left(y - 1\right) + x y - 1\right)^{2}} }{2 x^{2}},
  \label{g.f.UD}
  \end{equation}
  \begin{equation}
    \nonumber
    F_{UU}(x,y) = F_{DD}(x,y) = \frac{x^{2} y^{2} - x^{2} y - x + 1 - \sqrt{- 4 x^{2} y^{2} + \left(x^{2} y \left(y - 1\right) - x + 1\right)^{2}} }{2 x^{2} y^{2}},
    \label{g.f.UU}
  \end{equation}
  \begin{equation}
    \nonumber
    F_{DU}(x,y) = \frac{x^{2} y - x^{2} - x y + 1 - \sqrt{- 4 x^{2} + \left(x^{2} \left(y - 1\right) + x y - 1\right)^{2}} }{2 x^{2} y}.
    \label{g.f.DU}
  \end{equation}
  \label{thm.gf}
\end{thm}

\begin{proof}
  For $p=UD$, Corollary 3 in~\cite{decreasing-dyck} provides directly the bivariate generating function as a solution of a functional equation on Motzkin paths with respect to the number of occurrences of patterns $F$ and $UD$ (see~\cref{UD}).

For $p=UU$, there are two ways to obtain the generating function. First, we know that $\m{UU} + \m{UD} =\m n $ for Dyck paths of semilength $n$. So, we can obtain directly the generating function by calculating $F_{UD}(xy,\frac{1}{y})$. Moreover we know that $\phi(\m{UU})=\m U + \m{UU} +\m{UF}$ (see~\cref{UU}).  So, we decompose the set $\mathcal{M}$ of all Motzkin paths in the following way:  $$\mathcal{M} = \epsilon + F \mathcal{M} + UD \mathcal{M} + U
  (\mathcal{M \backslash \epsilon}) D \mathcal{M}.$$
	
The generating function for $\epsilon + F\mathcal{M}$ is given by $1+xM(x,y)$; the g.f. for $UD \mathcal{M}$ is $x^2yM(x,y)$ since $UD$ contains one occurrence of $U$; the g.f. for  $U
  (\mathcal{M \backslash \epsilon}) D \mathcal{M}$ is $x^2y^2(M(x,y)-1)M(x,y)$ since  a path in $U
  (\mathcal{M \backslash \epsilon})D$  starts with $U$ and its first two steps are  either $UU$ or  $UF$. This induces the functional
  equation
  $$ M(x,y) = 1 + xM(x,y) + x^2yM(x,y) + x^2y^2(M(x,y)-1)M(x,y), $$ which also gives the expected result.

For $p=DU$, the result is obtained by using the
  equality $\m{DU} = \m{UD} - \m 1$  and  evaluating $F_{DU}(x,y)=1+\frac{F_{UD}(x,y)-F_{UD}(x,0)}{y}$. Note that we can obtain this result using ~\cref{DU}.
\end{proof}

Table~\ref{tab.UD} gives the number of paths in $\mathcal{D}_n^{h,\geq}$ having $k$
    peaks $UD$ for $1\leq n\leq 11$ and $1\leq k\leq 8$. Applying standard techniques from generating function theory, we verify that the second row  corresponds to the sequence of quarter-squares numbers,
$\lfloor (n^2/4) \rfloor$, which is the  third row of Lozani\'c's triangle (see \sloane{A034851} and \cite{losanitsch1897}).
The third row is a shift of the  sequence \sloane{A005994} corresponding to alkane numbers $l(7, n)$
from fifth row of Lozani\'c's triangle, which enumerates certain symmetries exhibited by chemical
entities (alkane) consisting of hydrogen and carbon atoms arranged in a
tree-like structure. Third diagonal
corresponds to the squares, while fourth diagonal generates octahedral
numbers, $n(2n^2 + 1)/3$ (see \sloane{A005900}).

\begin{table}[H]
  \centering
  \small
  \begin{tabular}{c|ccccccccccc}

    $ k \backslash n $ &  1 &  2 &  3 &  4 &  5 &  6 &  7 &  8 &  9 &  10 &  11 \\\hline
     1  & 1 & 1 & 1 & 1 & 1 & 1  & 1  & 1   & 1   & 1   & 1        \\
     2  &   & 1 & 2 & 4 & 6 & 9  & 12 & 16  & 20  & 25  & 30      \\
     3  &   &   & 1 & 3 & 9 & 19 & 38 & 66  & 110 & 170 & 255    \\
     4  &   &   &   & 1 & 4 & 16 & 44 & 111 & 240 & 485 & 900   \\
     5  &   &   &   &   & 1 & 5  & 25 & 85  & 260 & 676 & 1615  \\
     6  &   &   &   &   &   & 1  & 6  & 36  & 146 & 526 & 1602  \\
     7  &   &   &   &   &   &    & 1  & 7   & 49  & 231 & 959   \\
     8  &   &   &   &   &   &    &    & 1   & 8   & 64  & 344   \\
     9  &   &   &   &   &   &    &    &     & 1   &  $\ldots$  & $\ldots$     \\
    \hline
    \bf $\Sigma$
           & 1 & 2 & 4 & 9 & 21 & 51 & 127 & 323 & 835 & 2188 & 5798
  \end{tabular}
  \caption{Number of paths in $\mathcal{D}_n^{h,\geq}$ having $k$
    peaks $UD$, or equivalently $k-1$ valleys $DU$, or equivalently $n - k$
    double rises $UU$.  }
\label{tab.UD}
\end{table}

\begin{cor}
  For $n\geq 0$, the popularity of pattern $p \in \{UU, UD, DD, DU\}$ in $\mathcal{D}_n^{h, \geq}$ is given by
   the generating function $G_{p}(x)$:

  \begin{equation}\nonumber
    G_{UD}(x) = \frac{\left(x - 1\right) \sqrt{- 3 x^{2} - 2 x + 1} - 3 x^{2} - 2 x + 1}{2 x \left(3 x - 1\right)}
    \label{p.UD},
  \end{equation}
  \begin{equation}\nonumber
    G_{UU}(x) = G_{DD}(x) = \frac{ \sqrt{- 3 x^{2} - 2 x + 1} \left(x^{2} + 2 x - 2\right) + x^{3} - 3 x^{2} - 4 x +2}{2 x^{2} \sqrt{- 3 x^{2} - 2 x + 1}},
    \label{p.UU}
  \end{equation}
  \begin{equation}\nonumber
    G_{DU}(x) = \frac{ \left(x^{2} - 1\right) \sqrt{- 3 x^{2} - 2 x + 1} - x^{3} - 3 x^{2} - x + 1}{2 x^{2} \sqrt{- 3 x^{2} - 2 x + 1}}.
    \label{p.DU}
  \end{equation}
  \label{cor1}
\end{cor}

\begin{proof}
  The generating function $G_{p}(x)$  of popularity is directly deduced from the
  bivariate generating function of pattern distribution
  $$
  G_{p}(x) = \frac{\partial F_{p}(x,y)}{\partial y} \bigg\vert_{y = 1}.
  $$
\end{proof}

See  Table~\ref{tab.pop2} for the first terms of the  popularity sequences.
The popularity of the pattern $UD$ generates a shift  of the sequence \sloane{A025566}
in~\cite{sloane}.  As suggested in~\cite{sloane}, the same sequence
enumerates the first differences of the directed animals~sequence~\sloane{A005773},
and also Motzkin paths of length $2n$ whose last
weak valley occurs immediately after step $n$.

The popularity sequence for $DU$ is the sequence \sloane{A025567}.  As mentioned
in~\cite{ferrari} by Ferrari and Munarini, this sequence corresponds
to the number of edges in Hasse diagram of Motzkin paths, where the
partial order is defined by the coverings $FF \mapsto UD, \; FU \mapsto UF, \; DF \mapsto FD, \; DU \mapsto
FF.$ This is an immediate consequence of the fact that $\phi$ maps pattern $DU$ from
constrained Dyck paths to the  patterns $FF, FU, DF, DU$ in  Motzkin paths.

The popularity sequence for $UU$ (or $DD$) does not yet appear in~\cite{sloane}.


\begin{table}[H]
  \begin{tabular}{c|c|c}
    Pattern &  Popularity sequence & OEIS \\\hline

    UD & $1,3,8,22,61,171,483,1373,3923, 11257, 32418, 93644 $ &
    \sloane{A025566} \\

    DU & $0,1, 4, 13, 40, 120, 356, 1050, 3088,
    9069, 26620, 78133$ & \sloane{A025567} \\

    UU, DD & $0,1, 4, 14, 44,
    135, 406, 1211, 3592, 10623, 31260, 92488$ &  \\
  \end{tabular}
  \caption{Popularity of length two patterns in $\mathcal{D}_n^{h, \geq}$ for $1\leq n\leq 12$.  }
  \label{tab.pop2}
\end{table}

\section{Patterns of length 3}
\label{3}

In this part we investigate how $\phi$ transports statistics
associated to patterns of length three, and for each of them we provide generating functions
for the distribution and popularity.

We need the following notations. For a step $X\in\{U,D,F\}$, an occurrence of the pattern $X^+$ inside a path $P$ is an occurrence of $X^k$, where  $X^k$ consists of $k$ consecutive repetitions  of $X$, for some $k\geq 1$. The associated statistic of $X^+$ (denoted $\m X^+$) will be equal to $\sum_{k\geq 1} \m X^k$ where $\m X^k$ is the statistic giving the number of occurrences of $X^k$. More generally, for two possibly empty sequences of steps $Y$ and $Z$ we define the pattern $YX^+Z$ as patterns of the form $YX^kZ$ for $k\geq 1$, and its associated statistic as $\m{YX^+Z}=\sum_{k\geq 1} \m{YX^kZ}$. For instance, an occurrence of the pattern $DU^+D$ can be an occurrence of $DUD$, $DUUD$, $DUUUD$, and so on... In the path $U UF\bm{D UUD}F\bm{D UDUUUD}DFFD D$ we count three occurrences of the pattern $DU^+D$. A path contains a dotted pattern $\e Y$
(resp. $Y\e$) if and only if it starts (resp. ends) with $Y$, and we use the notations $\e\m Y$, $\m Y\e$ for the associated statistics.

Any pattern $UU$ in a Dyck path is immediately followed by an up step or a down step, implying the statistic equation $\m{UU}=\m{UUU}+\m{UUD}$. Also, any pattern $UU$ in a Dyck path $P$ is either at the beginning of $P$ or  immediately preceded by an up or a down step, so $\m{UU}  = \e\m{UU} + \m{UUU} + \m{DUU}$. Using similar arguments we obtain the two following systems of statistic equations.

\begin{center}
$
(a)\left\lbrace\begin{aligned}
  \m{UU} & = \m{UUU} + \m{UUD}          \\
  \m{UU} & = \m{UUU} + \m{DUU} + \e\m{UU}  \\
  \m{UD} & = \m{UUD} + \m{DUD} + \e\m{UD}  \\
  \m{DU} & = \m{DUU} + \m{DUD}
\end{aligned}\right.,
$
\hspace*{0.5em}
$(b)\left\lbrace
\begin{aligned}
  \m{DD} & = \m{DDD} + \m{UDD}          \\
  \m{DD} & = \m{DDD} + \m{DDU} + \m{DD} \e  \\
  \m{UD} & = \m{UDD} + \m{UDU} + \m{UD} \e  \\
  \m{DU} & = \m{DDU} + \m{UDU}
\end{aligned}\right..
$
\end{center}

Observe that for any $P\in\mathcal{D}_n^{h,\geq}$, $n\geq 1$, we have $\e\m{UD}(P)=1$  (resp. $\e\m{UU}(P)=0$) when $P= (UD)^n$, and $\e\m{UD}(P)=0$ (resp. $\e\m{UU}(P)=1$) otherwise. So, we have the statistic equations $\phi(\e\m{UD})=\m\delta_{{F}^n}$ and $\phi(\e\m{UU})=1-\phi(\e\m{UD})$, where $\m\delta_{{F}^n}$ is the Dirac statistic defined by $\m\delta_{{F}^n}(P)=1$ whenever $P=F^n$, and $0$ otherwise.

Knowing the image through $\phi$ of only one statistic $\m{X}\in\{ \m{UUU}, \m{UUD}, \m{DUU}, \m{DUD}\}$ and using results from Section 2, we can obtain the expressions of the images of the three other statistics from the system $(a)$. The same reasoning holds for the second  system $(b)$. So, we split Section 3 into two subsections, each dealing with one of the two systems.


\subsection{The patterns $UUU, UUD, DUU, DUD$}

\begin{thm} For $n\geq 0$, the bijection $\phi$ from $\mathcal{D}_n^{h,\geq}$ to $\mathcal{M}_n$ transports the statistic $\m{UUD}$ as follows: $$ \phi (\m{UUD}) = \m{UF^+D}+\m{UD}.$$
  \label{UUD}
\end{thm}

\begin{proof} We proceed by induction on $n$. For $n=1$, we have $\m{UUD}(UD)=0$. Since $\phi(UD)=F$, we obtain $\m{UF^+D}(F)+\m{UD}(F)=0$, and the result holds. We assume the result for $k\leq n$, and we will prove it for $n+1$ by distinguishing two main cases.

- Whenever $P=\alpha UD$, we have $\phi(P)=\phi(\alpha)F$, and $(\m{UF^+D}+\m{UD})(\phi(P))=(\m{UF^+D}+\m{UD})(\phi(\alpha))$. Using induction hypothesis, $(\m{UF^+D}+\m{UD})(\phi(\alpha))=\m{UUD}(\alpha)$ which is also equal to $\m{UUD}(\alpha UD))$, proving the first case.

- Whenever $P=\alpha UU\beta D \gamma D$, we have $\phi(P)=\phi(\alpha)\phi(\gamma)U\phi(\beta)D$, and  $(\m{UF^+D}+\m{UD})(\phi(P))=(\m{UF^+D}+\m{UD})(\phi(\alpha))+(\m{UF^+D}+\m{UD})(\phi(\gamma)U\phi(\beta)D)$. Using induction hypothesis, we obtain
$(\m{UF^+D}+\m{UD})(\phi(P))=\m{UUD}(\alpha)+(\m{UF^+D}+\m{UD})(\phi(\gamma)U\phi(\beta)D)$. Now we distinguish three subcases in order to prove that $(\m{UF^+D}+\m{UD})(\phi(\gamma)U\phi(\beta)D)=\m{UUD}(UU\beta D\gamma D)$.

If $\beta$ is the empty path, then $(\m{UF^+D}+\m{UD})(\phi(\gamma)U\phi(\beta)D)=1+(\m{UF^+D}+\m{UD})(\phi(\gamma))=1+\m{UUD}(\gamma)=\m{UUD}(UUD\gamma D)$.

If $\beta=(UD)^k$ for some $k\geq 1$, then $(\m{UF^+D}+\m{UD})(\phi(\gamma)U\phi((UD)^k)D)=(\m{UF^+D}+\m{UD})(\phi(\gamma)UF^kD)=(\m{UF^+D}+\m{UD})(\phi(\gamma))+1
= \m{UUD}(\gamma)+1=\m{UUD}(UU(UD)^kD\gamma D)=\m{UUD}(UU\beta D\gamma D)$.

If $\beta$ starts with a double rise $UU$, then $\phi(\beta)$ contains at least one up step ($\phi(\beta)\neq F^\ell$ for any $\ell\geq 0$). So, we have
$(\m{UF^+D}+\m{UD})(\phi(\gamma)U\phi(\beta)D) =
(\m{UF^+D}+\m{UD})(\phi(\gamma))+(\m{UF^+D}+\m{UD})(\phi(\beta)) =
\m{UUD}(\gamma)+\m{UUD}(\beta)$, which equals to $\m{UUD}(UU\beta  D \gamma D)$.

Considering these three cases, the second case is proved and the induction is completed.\end{proof}


\begin{thm} For $n\geq 0$, the bijection $\phi$ from $\mathcal{D}_n^{h,\geq}$ to $\mathcal{M}_n$ transports the statistic $\m{UUU}$ as follows:
  $$ \phi (\m{UUU}) = \m{UF^+D} + 2 \big( \m{UF^+U} +\m{UU}\big).$$
  \label{UUU}
\end{thm}
\begin{proof} Using Theorem~\ref{UUD} and the equation $\m{UU} = \m{UUD} + \m{UUU}$ of system $(a)$, we have
  $\phi(\m{UUU}) = \phi(\m{UU}) - \phi(\m{UUD})=\phi(\m{UU})- \m{UF^+D}-\m{UD}$,
  and using Theorem~\ref{thm1}, we have
  $\phi(\m{UUU}) = \m{U} + \m{UU} + \m{UF} - \m{UF^+D}-\m{UD}.$
 Since  we have $\m{U}=\m{UD}+\m{UU}+\m{UF}$ on $\mathcal{M}_n$, we obtain
  $\phi(\m{UUU}) = 2 (\m{UU} + \m{UF} )  -  \m{UF^+D}.$
Using the trivial equation $\m{UF}=\m{UF^+D}+\m{UF^+U}$ for Motzkin paths, we complete the proof.
\end{proof}

\begin{thm}
  For $n\geq 1$, the bijection $\phi$ from $\mathcal{D}_n^{h,\geq}$ to $\mathcal{M}_n$ transports the statistic $\m{DUU}$ as follows:
  $$ \phi (\m{DUU}) = \m{UF^+D} +\m{UD} + \m\delta_{F^n}-\m{1}, $$
where $\m\delta_{F^n}$ is the Dirac statistic defined by $\m\delta_{F^n}(P)=1$ whenever $P=F^n$, and $0$ otherwise.  \label{DUU}
\end{thm}

\begin{proof}
  Using the difference of the first two equations of $(a)$, we obtain   $\m{DUU} = \m{UUD} - \e\m{UU}$, and finally,
  $\phi(\m{DUU})=\phi(\m{UUD})-\phi(\e\m{UU})= \m{UF^+D} +\m{UD} -\phi(\e\m{UU})$. As discussed at the beginning of Section 3, $\phi(\e\m{UU})=1-\m\delta_{F^n}$, which completes the proof.
\end{proof}

\begin{thm}
  For $n\geq 1$, the bijection $\phi$ from $\mathcal{D}_n^{h,\geq}$ to $\mathcal{M}_n$ transports the statistic $\m{DUD}$ as follows:
    $$
  \phi (\m{DUD})  = \m{F} - \m{UF^+D}-\m\delta_{F^n}
  $$
  \label{DUD}
\end{thm}

\begin{proof}
  We have $\m{DUD}= \m{UD} - \m{UUD} - \bullet\m{UD}$, which implies that $\phi(\m{DUD})=\phi(\m{UD})-\phi(\m{UUD})-\phi(\bullet\m{UD})$.
  From Theorem \ref{thm1}, Theorem \ref{UUD} and the observations after the systems ($a$) and ($b$) we obtain $\phi(\m{DUD})=\m{UD}+\m{F}-\m{UD}-\m{UF^+D}-\m\delta_{F^n}=\m{F}-\m{UF^+D}-\m\delta_{F^n}$.
\end{proof}

\begin{thm}
  The bivariate generating functions $F_p(x,y)$ where the coefficient of $x^ny^k$ is the number of Dyck paths in
   $\mathcal{D}_n^{h,\geq}$ containing exactly $k$ occurrences of the pattern $p\in\{UUU,UUD,DUU,DUD\}$ are given by the following expressions listed in the same order as written above.
  \scriptsize

  \begin{equation}
    \nonumber\frac{x^3 y - x^3 - x^2 y^2 + 2 x - 1 + \sqrt{(x - x y - 1) (x^2 - x y + x - 1)
        ( x^3 - x^3 y + x^2 y^2 + 2 x^2 y - 2 x y - 2 x + 1)} } {2 x^2 y^2 (x - 1)},
    \label{g.f.UUU}
  \end{equation}
  \begin{equation}
    \nonumber
     \frac{x^{2} y - 2 x^{2} + 2 x - 1 + \sqrt{\left(x^{2} y - 1\right) \left(x^{2} y - 4 x^{2} + 4 x - 1\right)} }{2 x^{2} \left(x - 1\right)},
    \label{g.f.UUD}
  \end{equation}
  \begin{equation}
    \nonumber
     \frac{2 x - x^{2} y -1 + \sqrt{\left(x^{2} y - 1\right) \left(x^{2} y - 4 x^{2} + 4 x - 1\right)}}{2 x^{2} y \left(x - 1\right)},
    \label{g.f.DUU}
  \end{equation}
  \begin{equation}
    \nonumber
    \frac{
      x^3 y - x^3  - x^2 y^2 + 2 x y
      - 1 +
      \sqrt{
        (xy - x - 1)
        (x^2 + yx - x - 1)
        (x^3 y - x^3 + x^2 y^2 + 2 x^2 y - 2 x y - 2 x + 1)
      }
    }{2 x^2 ( x y - 1)}.
    \label{g.f.DUD}
  \end{equation}
  \label{UUU.UUD.DUU.DUD.g.f}
\end{thm}

\begin{proof} For $p=UUU$ and using Theorem \ref{UUU}, we have $\phi(\m{UUU})=\m{UF^+D}+2(\m{UF^+U}+\m{UU})$. So, we decompose Motzkin paths according to the patterns $UF^+D$, $UF^+U$, and $UU$ in order to exhibit a functional equation having $F_p(x,y)$ as solution:
$$\mathcal{M}=\epsilon + F \mathcal{M} + UD \mathcal{M}+ U\mathcal{M}_0 D\mathcal{M}+U(\mathcal{M}\backslash(\mathcal{M}_0\cup\epsilon))D\mathcal{M},$$ where
$\mathcal{M}_0$ is the subset of $\mathcal{M}$ consisting of paths of the form $F^k$ for $k\geq 1$.
The generating function for $\epsilon+F\mathcal{M}+UD \mathcal{M}$ is $1 + x F_{UUU}(x,y) + x^2 F_{UUU}(x,y)$; the g.f. for $U\mathcal{M}_0 D\mathcal{M}$ is $x^2y \frac{x}{1-x}
  F_{UUU}(x,y)$ since a path in $U\mathcal{M}_0 D$ is an occurrence of $UF^+D$; the g.f. for  $U(\mathcal{M}\backslash(\mathcal{M}_0\cup\epsilon))D\mathcal{M}$ is $x^2 y^2 \left( F_{UUU}(x,y) - \frac{1}{1-x} \right)
  F_{UUU}(x,y)$ since a path in $U(\mathcal{M}\backslash(\mathcal{M}_0\cup\epsilon))D$ starts with either an occurrence of $UU$ or an occurrence of $UF^+U$.
	
So the functional equation is:
$$\begin{aligned}
  F_{UUU}(x,y) = 1 + x F_{UUU}(x,y) +& x^2 F_{UUU}(x,y) + x^2y \frac{x}{1-x}
  F_{UUU}(x,y) +\\
  +& x^2 y^2 \left( F_{UUU}(x,y) - \frac{1}{1-x} \right)
  F_{UUU}(x,y) .
\end{aligned}$$
  A simple calculation (with Maple for instance) provides the result.

All other generating functions are obtained using a  similar method, so we do not give the proofs here.
\end{proof}

The generating function $G_p(x)$ for the popularity of  the pattern $p\in\{UUU,UUD,$ $DUU, DUD\}$ is obtained directly by evaluating $\frac{\partial F_{p}(x,y)}{\partial y} \big\vert_{y = 1}.$ Table~\ref{tab.pop3a} provides the first terms of the associated sequences. See also Table~\ref{UUU.DUD.UUD.DUU.tab} for an illustration of the distribution of $p\in\{UUU,UUD,DUU, DUD\}$.

\begin{table}[H]
  \begin{tabular}{c|c|c}
    Pattern &  Popularity sequence & OEIS \\\hline

    UUU &  $0,0,1, 5, 19, 65, 210, 658, 2023, 6147, 18534, 55594 $ &  \\
    UUD &  $0,1, 3, 9, 25, 70, 196, 553, 1569, 4476, 12826, 36894$ & \sloane{A097861}  \\
    DUU & $0,0,0,1, 5, 20, 70, 231, 735, 2289, 7029, 21384$  & \sloane{A304011} ? \\
    DUD & $0,1, 4, 12, 35, 100, 286, 819, 2353, 6780, 19591, 56749$ &    \\
  \end{tabular}
  \caption{Popularity of $p\in\{UUU,UUD,DUU,DUD\}$ in $\mathcal{D}_n^{h, \geq}$ for $1\leq n\leq 12$.}
  \label{tab.pop3a}
\end{table}

The popularity sequence for $UUD$ in $\mathcal{D}_n^{h,\geq}$ is
the sequence \sloane{A097861}, corresponding to the number of humps in
all Motzkin paths of length $n$ (a hump equals $UF^+D$ or $UD$ in our
notation).  The popularity for $DUU$ seems to generate the sequence \sloane{A304011}, but we did not succeed
 in proving this fact, so we leave it as a conjecture.

\setlength{\tabcolsep}{0.35em}
\begin{table}[H]
  \centering
  \small
  \begin{subfigure}{0.48\textwidth}
    \begin{tabular}{c|ccccccccc}
    $ k \backslash n $ &   1 &  2 &  3 &  4 &  5 &  6 &  7 &  8 &  9   \\\hline
            0	& 1	& 2	& 3	& 5	& 8	& 13	& 21	& 34	& 55		 \\
            1	& 	& 	& 1	& 3	& 8	& 18	& 38	& 76	& 147		 \\
            2	& 	& 	& 	& 1	& 4	& 14	& 40	& 104	& 250		 \\
            3	& 	& 	& 	& 	& 1	& 5	& 21	& 71	& 215		 \\
            4	& 	& 	& 	& 	& 	& 1	& 6	& 30	& 119		 \\
            5	& 	& 	& 	& 	& 	& 	& 1	& 7	& 40		 \\
  \end{tabular}
  \caption{$UUU$}
  \label{UUU.tab}
  \end{subfigure}
  \quad
  \begin{subfigure}{0.48\textwidth}
    \begin{tabular}{c|ccccccccc}
      $ k \backslash n $ &   1 &  2 &  3 &  4 &  5 &  6 &  7 &  8 &  9  \\\hline
       0 & 1	& 1	& 1	& 2	& 3	& 6	& 10	& 20	& 36	  \\
       1 & 	& 1	& 2	& 3	& 7	& 13	& 30	& 58	& 130	   \\
       2 & 	& 	& 1	& 3	& 6	& 16	& 35	& 91	& 199	   \\
       3 & 	& 	& 	& 1	& 4	& 10	& 30	& 75	& 216	  \\
       4 & 	& 	& 	& 	& 1	& 5	& 15	& 50	& 140	   \\
       5 & 	& 	& 	& 	& 	& 1	& 6	& 21	& 77	   \\
    \end{tabular}
    \caption{$DUD$}
    \label{DUD.tab}
  \end{subfigure}

  \vspace*{1em}

  \begin{subfigure}{0.48\textwidth}
   \begin{tabular}{c|ccccccccc}
      $ k \backslash n $ &  1 &  2 &  3 &  4 &  5 &  6 &  7 &  8 &  9 \\\hline
       0 & 1 & 1 & 1 & 1 & 1 & 1 & 1 & 1 & 1 \\
       1 &  & 1 & 3 & 7 & 15 & 31 & 63 & 127 & 255 \\
       2 &  &  &  & 1 & 5 & 18 & 56 & 160 & 432 \\
       3 &  &  &  &  &  & 1 & 7 & 34 & 138 \\
       4 &  &  &  &  &  &  &  & 1 & 9 \\
    \end{tabular}
    \caption{$UUD$}
    \label{UUD.tab}
  \end{subfigure}
  \quad
  \begin{subfigure}{0.48\textwidth}
    \begin{tabular}{c|ccccccccc}
      $ k \backslash n $ &  1 &  2 &  3 &  4 &  5 &  6 &  7 &  8 &  9\\\hline
       0 & 1 & 2 & 4 & 8 & 16 & 32 & 64 & 128 & 256 \\
       1 &  &  &  & 1 & 5 & 18 & 56 & 160 & 432 \\
       2 &  &  &  &  &  & 1 & 7 & 34 & 138 \\
       3 &  &  &  &  &  &  &  & 1 & 9 \\
        4 &  &  &  &  &  &  &  &  &  \\
   \end{tabular}
    \caption{$DUU$}
    \label{DUU.tab}
  \end{subfigure}

  \caption{Number of paths from $\mathcal{D}_n^{h,\geq}$ having $k$
    occurrences of the considered  pattern. }
\label{UUU.DUD.UUD.DUU.tab}
\end{table}
Dyck paths from $\mathcal{D}_n^{h,\geq}$ avoiding $UUU$, $DUU$ respectively generate
 Fibonacci numbers and integer squares. Those avoiding $DUD$ seem to correspond to
\sloane{A007562} (number of planted trees where non-root, non-leaf
nodes at even distance from root are of degree 2). At the present
time, there is no closed form for the generating function of sequence
\sloane{A007562}. Note that any path avoiding $DUU$ has at
    most one occurrence of $UUD$. Also, Dyck paths containing two occurrences of $UUD$ in $\mathcal{D}_n^{h,\geq}$
    generate a shift of sequence \sloane{A001793}, which corresponds to a subsequence in the triangle of coefficients of Chebyshev's polynomials which is sequence \sloane{A053120}.

\subsection{The patterns $DDD, DDU, UDD, UDU$}

\begin{thm} For $n\geq 0$, the bijection $\phi$ from $\mathcal{D}_n^{h,\geq}$ to $\mathcal{M}_n$ transports the statistic  $\m{UDU}$ as follows:
   $$\phi (\m{UDU})  = \m{FF} + \m{FUD} $$
 \label{UDU}\end{thm}

\begin{proof} We proceed by induction on $n$. For $n=1$, we have $\m{UDU}(UD)=0$. With $\phi(UD)=F$, we obtain $\m{FF}(F)+\m{FUD}(F)=0$, and the result holds. We assume the result for $k\leq n$,
and we will prove it for $n+1$.

 - Whenever $P = \alpha UD$, we have
$\phi(P)=\phi(\alpha)F$, and $(\m{FF}+\m{FUD})(\phi(P))=\m{FF}(\phi(\alpha)F)+\m{FUD}(\phi(\alpha)F)=\m{FF}(\phi(\alpha)F)+\m{FUD}(\phi(\alpha))$.
We distinguish two cases: ($i$) $\phi(\alpha)$ ends with $F$, and ($ii$) otherwise.
In the case ($i$), $\alpha$ ends with $UD$, and thus $(\m{FF}+\m{FUD})(\phi(P))=1+\m{FF}(\phi(\alpha))+\m{FUD}(\phi(\alpha))$.
Using the induction hypothesis we have
  $(\m{FF}+\m{FUD})(\phi(P))=1+\m{UDU}(\alpha)=\m{UDU}(\alpha UD)=\m{UDU}(P)$.
In the case ($ii$), $\alpha$ does not end with $UD$, and thus $\phi(\alpha)$ does not end with $F$. So, we have $(\m{FF}+\m{FUD})(\phi(P))=\m{FF}(\phi(\alpha))+\m{FUD}(\phi(\alpha))$, and using the induction  hypothesis
$(\m{FF}+\m{FUD})(\phi(P))=\m{UDU}(\alpha)=\m{UDU}(\alpha UD)=\m{UDU}(P)$.

- Whenever $P$ does not end with $UD$, we have $P=\alpha UU \beta D \gamma D$,  and $(\m{FF}+\m{FUD})(\phi(P))=(\m{FF}+\m{FUD})(\phi(\alpha)\phi(\gamma)U\phi(\beta)D)$.
Note that $\alpha$ cannot end with $UD$, otherwise it would contradict $P\in\mathcal{D}_n^{h,\geq}$. This means that $\phi(\alpha)$ cannot end with $F$.
So, all the possible occurrences of $FF$ in $\phi(P)$ belong necessarily to $\phi(\alpha)$, $\phi(\beta)$ and $\phi(\gamma)$, which implies that $\m{FF}(\phi(P))=\m{FF}(\phi(\alpha))+\m{FF}(\phi(\beta))+\m{FF}(\phi(\gamma))$.
On the other hand, the possible occurrences of $FUD$ in $\phi(P)$ belong necessarily to $\phi(\alpha)$, $\phi(\beta)$, $\phi(\gamma)$, and eventually at the junction of $\phi(\gamma)$ and $U\phi(\beta)D$  whenever $\phi(\gamma)$ ends with $F$ and $\phi(\beta)=\epsilon$. So, we distinguish two cases: ($a$) $\phi(\gamma)$ ends with $F$ and $\beta=\epsilon$, and $(b)$ otherwise.
In the case ($a$), we have $\m{FUD}(\phi(P))=1+\m{FF}(\phi(\alpha))+\m{FF}(\phi(\beta))+\m{FF}(\phi(\gamma))+\m{FUD}(\phi(\alpha))+\m{FUD}(\phi(\beta))+\m{FUD}(\phi(\gamma))$, and using the induction  hypothesis, we obtain $\m{FUD}(\phi(P))=1+\m{UDU}(\alpha)+\m{UDU}(\beta)+\m{UDU}(\gamma)=\m{UDU}(\alpha UU \beta D \gamma D)=\m{UDU}(P)$.
In the case ($b$), we have $\m{FUD}(\phi(P))=\m{FF}(\phi(\alpha))+\m{FF}(\phi(\beta))+\m{FF}(\phi(\gamma))+\m{FUD}(\phi(\alpha))+\m{FUD}(\phi(\beta))+\m{FUD}(\phi(\gamma))$,
and using the induction hypothesis, we obtain $\m{FUD}(\phi(P))=\m{UDU}(\alpha)+\m{UDU}(\beta)+\m{UDU}(\gamma)=\m{UDU}(\alpha UU \beta D \gamma D)=\m{UDU}(P)$.
So, the proof is complete.
\end{proof}

\begin{thm} For $n\geq 0$, the bijection $\phi$ from $\mathcal{D}_n^{h,\geq}$ to $\mathcal{M}_n$ transports the statistic  $\m{UDD}$ as follows:
$$ \phi (\m{UDD})  = \m{FD} + \m{UD} + \m{FUU} + \m{FUF}.$$
\label{UDD}
\end{thm}
\begin{proof} Considering the third equation of system ($b$), we have $\phi(\m{UDD}) = \phi(\m{UD}) - \phi(\m{UDU}) - \phi(\m{UD}\bullet)$. Using Theorems \ref{thm1} and \ref{UDU}, we obtain $\phi(\m{UDD}) = \m{F}+\m{UD} -\m{FF}-\m{FUD} - \phi(\m{UD}\bullet)$. In any  Motzkin path $P$, a flat step is either at the end of $P$, or followed by $F$, or $D$, or $UU$, or $UD$, or $UF$, that is $\m{F} = \m{F}\bullet + \m{FF} + \m{FD}  + \m{FUU} + \m{FUD} + \m{FUF}$. Then, we obtain $\phi(\m{UDD})=\m{F}\bullet +\m{UD} + \m{FD} + \m{FUU} + \m{FUF} - \phi(\m{UD}\bullet)$. Using the definition of $\phi$  in Introduction, it is clear that  $\phi$ transports $\m{UD}\bullet$ into $\m{F}\bullet$, and the result holds.
\end{proof}

\begin{thm} For $n\geq 0$, the bijection $\phi$ from $\mathcal{D}_n^{h,\geq}$ to $\mathcal{M}_n$ transports the statistic  $\m{DDU}$ as follows:
$$ \phi(\m{DDU})  = \m{DF} + \m{DU} + \m{FUU} + \m{FUF}.$$
\label{DDU}\end{thm}
\begin{proof} Combining the fourth equation $\m{DU} = \m{DDU}  + \m{UDU}$ of system ($b$) with  Theorems \ref{thm1} and \ref{UDU}, we obtain $\phi(\m{DDU})=\m{FF}+\m{FU}+\m{DF}+\m{DU}-\m{FF}-\m{FUD}=\m{FU}+\m{DF}+\m{DU}-\m{FUD}$. The proof is completed using the straightforward equation  $\m{FU}=\m{FUU}+\m{FUF}+\m{FUD}$ on Motzkin paths.
\end{proof}

\begin{thm} For $n\geq 0$, the bijection $\phi$ from $\mathcal{D}_n^{h,\geq}$ to $\mathcal{M}_n$ transports the statistic  $\m{DDD}$ as follows:
$$\phi (\m{DDD})  = 2 (\m{UU} + \m{UF}) - \m{FD} - \m{FUU} - \m{FUF}.$$
\label{DDD}\end{thm}
\begin{proof} Combining the first equation $\m{DD} = \m{DDD}  + \m{UDD}$ of system ($b$) with  Theorems \ref{thm1} and \ref{UDD}, we obtain
 $\phi(\m{DDD})=\m{U}+\m{UU}+\m{UF}-\m{FD}-\m{UD}-\m{FUU}-\m{FUF}$. Using $\m{U}=\m{UF}+\m{UD}+\m{UU}$ on Motzkin paths, the result holds.
\end{proof}

\begin{thm} The bivariate generating functions $F_p(x,y)$ where the coefficient of $x^ny^k$ is the number of Dyck paths in $\mathcal{D}_n^{h,\geq}$
containing exactly $k$ occurrences of the pattern $p\in \{UDU,UDD,DDU,DDD\}$ are given by the following expressions listed in the same order as written above.

  $$\frac{1 + x (x^2 - x^2 y - y) - \sqrt{
      A
    }
  }{2 x^2 \left(x - x y + 1\right)},
  $$
  $$
  \frac{1 + x (x^2 y - x^2 - x y + x - 1) - \sqrt{B}}
  {2 x^2 \left(x y - x + 1\right)^2},
  $$
  $$
  \frac{1 + x ( 2 x^2 y^2 - 3 x^2 y + x^2 + x y - x - 1) - \sqrt{C
  }}
  {2 x^2 y \left(x y - x + 1\right)},
  $$
  $$
  \frac{1 - x \left(x^2 y^2 - x^2 y - x y^2 + x + 1\right) - \sqrt{D} }{2 x^2 \left(x y - x - y\right)^2}
  $$
  where
  $$A= \left(x + 1\right) \left(x^{2} y - x^{2} + x y - x - 1\right) \left(x^{3} y - x^{3} - 2 x^{2} y + 2 x^{2} + x y + 2 x - 1\right),$$
  $$B= \left(x + 1\right) \left(x^2 y - x^2 + 1\right) \left(x^3 y - x^3 - 3 x^2 y + 3 x^2 - 3 x + 1\right) ,$$
  $$C= \left(x + 1\right) \left(x^2 y - x^2 + 1\right) \left(x^3 y - x^3 - 3 x^2 y + 3 x^2 - 3 x + 1\right) ,$$
  $$D=\left(x y + 1\right) \left(x^2 y - x^2 - x y + x - 1\right) \left(x^3 y^2 - x^3 y - x^2 y^2 - 2 x^2 y + 3 x^2 + 2 x y + x - 1\right).$$
  \label{UDU.UDD.DDU.DDD.g.f}
\end{thm}

\begin{proof}

  Let $A_p(x,y)$ (resp. $B_p(x,y)$) be the bivariate generating
  function where the coefficient of $x^ny^k$ is the number of
  Motzkin paths of length $n$ ending with a flat step (resp. down step) having exactly $k$
  occurrences of the patterns related to the statistic  $\phi(\m{p})$ obtained in the  r.h.s. of equations
  of Theorems~\ref{UDU},\ref{UDD},\ref{DDU},\ref{DDD}.
  Using a refinement of the classical decompositions
  of  non-empty Motzkin paths by taking into account the occurrences of the considered patterns, we deduce functional equations
  for $A_p(x,y)$ and $B_p(x,y)$. The method being classic, we do not give any more details. The solutions are obtained  by a simple calculation.

  For $p=UDU$:

 \noindent $   \begin{cases}
       F_{p}(x,y) &= 1 + A_{p}(x,y) + B_{p}(x,y) \\
       A_{p}(x,y) &= x + x y A_{p}(x,y) + x B_{p}(x,y) \\
       B_{p}(x,y) &= x^2 + x^2 y A_{p}(x,y) + x^2 B_{p}(x,y) + x^2 F_{p}(x,y) (F_{p}(x,y) - 1);
     \end{cases} $

  For $p=UDD$:

 \noindent $  \begin{cases}
       F_{p}(x,y) = 1 + A_p(x,y) + B_p(x,y)        \\
       A_p(x,y) = xF_{p}(x,y)                 \\
       B_p(x,y) = x^2y F_{p}(x,y) + x^2y A_p(x,y) + x^2 B_p(x,y) +\\
       \quad\quad + x^2 B_p(x,y)^2 + x^2y A_p(x,y) B_p(x,y) + x^2y^2 A_p(x,y)^2 + x^2y A_p(x,y) B_p(x,y);
     \end{cases}
 $

 For $p=DDU$:

\noindent$ \begin{cases}
       F_{p}(x,y) = 1 + A_p(x,y) + B_p(x,y)                      \\
       A_p(x,y) = x  +  x A_p(x,y)  +  x y B_p(x,y)           \\
       B_p(x,y) = x^2 F_p(x,y)   +   x^2 y A_p(x,y) (F_{p}(x,y) - 1)    + x^2 A_p(x,y) +\\
         \qquad + x^2 y B_p(x,y) F_{p}(x,y);
     \end{cases}$

 For $p=DDD$:

\noindent$   \begin{cases}
       F_{p}(x,y) = 1 + A_p(x,y) + B_p(x,y)                       \\
       A_p(x,y) = x F_{p}(x,y)                               \\
       B_p(x,y) = x^2 F_{p}(x,y)  +  x^2 y^2 B_p(x,y)  + x^2 y^2 A_p(x,y) / y  +  x^2 y^2 A_p(x,y)^2 / y^2  + \\ \qquad +x^2 y^2 A_p(x,y) B_p(x,y) / y + x^2 y^2 B_p(x,y)^2 + x^2 y^2 A_p(x,y) B_p(x,y) / y
      .
   \end{cases}
    $

The above equations are intentionally left in non-simplified
  forms, in order to allow the reader to easily retrieve  the refined decompositions from the classical decomposition of Motzkin paths
   $\mathcal{M} = \epsilon + \mathcal{M} F
  + \mathcal{M} U \mathcal{M} D$.
\end{proof}

Generating function $G_p(x)$ for the popularity of  a pattern $p$  is obtained directly by evaluating $\frac{\partial F_{p}(x,y)}{\partial y} \big\vert_{y = 1}.$ Table~\ref{tab.pop3b} provides the first terms of the generated sequences. See also Table~\ref{DDD.UDU.DDU.UDD.tab} for an illustration of the distribution of $p\in\{UDU,UDD,DDU, DDD\}$.

 Unlike what happens for classical Dyck paths, the
popularity of $UDU$ in $\mathcal{D}_{n+1}^{h, \geq}$ is equal to the
popularity of $UD$ in $\mathcal{D}_n^{h,\geq}$, while the corresponding
distributions are different. Dyck paths from $\mathcal{D}_n^{h,\geq}$
avoiding a pattern $UDU$ (resp. $DDD$) are counted by the  Generalized Catalan numbers (resp. by the numbers of
ordered trees with $n$ edges and having no branches of length 1), which corresponds to the sequence \sloane{A004148} (resp. \sloane{A026418}).


\begin{table}[H]
  \centering \small\small
  \begin{subfigure}{0.45\textwidth}
   \begin{tabular}{c|cccccccccc}
      $ k \backslash n $ &  1 &  2 & 3 &  4 &  5 &  6 &  7 &  8 &  9 \\\hline
  0 &	1 &	2 &	3 &	6 &	11 &	22 &	43 &	87 &	176 \\
  1 &	 &	 &	1 &	2 &	7 &	16 &	43 &	102 &	251 \\
  2 &	 &	 &	 &	1 &	2 &	10 &	25 &	80 &	208  \\
  3 &	 &	 &	 &	 &	1 &	2 &	13 &	34 &	130 \\
  4 &	 &	 &	 &	 &	 &	1 &	2 &	17 &	46  \\
  5 &	 &	 &	 &	 &	 &	 &	1 &	2 &	21   \\
    \end{tabular}
    \caption{$DDD$}
    \label{DDD.tab}
  \end{subfigure}
  \quad
  \begin{subfigure}{0.45\textwidth}
    \begin{tabular}{c|cccccccccc}
      $ k \backslash n $ &  1 &  2 &  3 &  4 &  5 &  6 &  7 &  8 &  9 \\\hline
 0  &	1 &	1 &	2 &	4 &	8 &	17 &	37 &	82 &	185 \\
 1  &	 &	1 &	1 &	3 &	7 &	17 &	41 &	102 &	252  \\
 2  &	 &	 &	1 &	1 &	4 &	10 &	28 &	73 &	200 \\
 3  &	 &	 &	 &	1 &	1 &	5 &	13 &	41 &	113  \\
 4  &	 &	 &	 &	 &	1 &	1 &	6 &	16 &	56  \\
 5  &	 &	 &	 &	 &	 &	1 &	1 &	7 &	19  \\
    \end{tabular}
    \caption{$UDU$}
    \label{UDU.tab}
  \end{subfigure}
  \vspace*{1em}
  \begin{subfigure}{0.45\textwidth}
    \begin{tabular}{c|cccccccccc}
      $ k \backslash n $ &  1 &  2 &  3 &  4 &  5 &  6 &  7 &  8 &  9 \\\hline
       0 & 1	& 2	& 3	& 4	& 5	& 6	& 7	& 8	& 9	 \\
       1 & 	& 	& 1	& 5	& 14	& 31	& 59	& 102	& 164	 \\
       2 & 	& 	& 	& 	& 2	& 14	& 57	& 174	& 444	 \\
      3 & 	& 	& 	& 	& 	& 	& 4	& 39	& 209	 \\
       4 & 	& 	& 	& 	& 	& 	& 	& 	& 9	
    \end{tabular}
    \caption{$DDU$}
    \label{DDU.tab}
  \end{subfigure}
  \quad
  \begin{subfigure}{0.45\textwidth}
    \begin{tabular}{c|cccccccccc}
      $ k \backslash n $ & 1 &  2 &  3 &  4 &  5 &  6 &  7 &  8 &  9 \\\hline
          0 & 1	& 1	& 1	& 1	& 1	& 1	& 1	& 1	& 1		\\
         1 & 	& 1	& 3	& 6	& 10	& 15	& 21	& 28	& 36		\\
          2 & 	& 	& 	& 2	& 10	& 31	& 75	& 156	& 292		\\
         3 & 	& 	& 	& 	& 	& 4	& 30	& 129	& 417		\\
          4 & 	& 	& 	& 	& 	& 	& 	& 9	& 89		\\
    \end{tabular}
    \caption{$UDD$}
    \label{UDD.tab}
  \end{subfigure}
  \caption{Number of paths from $\mathcal{D}_n^{h,\geq}$ having $k$
    occurences of the considered pattern.}
  \label{DDD.UDU.DDU.UDD.tab}
\end{table}

\begin{table}[H]
  \begin{tabular}{c|c|c}
     Pattern &  Popularity sequence & OEIS \\\hline
    DDD & $0,0,1, 4, 14, 46, 145, 448, 1365, 4124, 12387, 37060$ &                          \\
    DDU & $0,0,1, 5, 18, 59, 185, 567, 1715, 5146, 15363, 45715$ &                          \\
    UDD & $0,1, 3, 10, 30, 89, 261, 763, 2227, 6499, 18973, 55428$ &                       \\
    UDU & $0,1, 3, 8, 22, 61, 171, 483, 1373, 3923, 11257, 32418$ &   \sloane{A025566}  \\
  \end{tabular}
  \caption{Popularity of 3-length patterns in $\mathcal{D}_n^{h, \geq}$ for $1\leq n \leq 12$.
  }
  \label{tab.pop3b}
\end{table}
\section{Acknowledgements} We would like to thank the anonymous referees for their helpful comments
and suggestions.

\bibliographystyle{plain}

\end{document}